\documentclass[11pt]{amsart}
\usepackage{amssymb,amsmath,amsthm,amscd}
\usepackage[all]{xy}

\numberwithin{equation}{section}

\newtheoremstyle{fancy1}{10pt}{10pt}{\itshape}{12pt}{\textsc\bgroup}{.\egroup}{8pt}{
}
\newtheoremstyle{fancy2}{10pt}{10pt}{}{12pt}{\itshape}{.}{8pt}{ }

\theoremstyle{fancy1}

\newtheorem{cor}[equation]{Corollary}
\newtheorem{lem}[equation]{Lemma}

\newtheorem{thm}[equation]{Theorem}

\newtheorem{main}{Theorem}
\newtheorem*{main*}{Theorem}

\newtheorem*{cor*}{Corollary}

\theoremstyle{fancy2}

\newtheorem{rem}[equation]{Remark}
\newtheorem*{rem*}{Remark}

\newcommand{\cref}[1]{Corollary~\ref{#1}}

\newcommand{\Sph}{\mathbb{S}}
\newcommand{\Disc}{\mathbb{D}}

\newcommand{\gS}{\mathsf{S}}

\newcommand{\eps}{\varepsilon}

\newcommand{\RP}{\mathbb{R\mkern1mu P}}
\newcommand{\CP}{\mathbb{C\mkern1mu P}}

\newcommand{\C}{{\mathbb{C}}}
\newcommand{\R}{{\mathbb{R}}}
\newcommand{\Z}{{\mathbb{Z}}}

\newcommand{\SO}{\ensuremath{\operatorname{\mathsf{SO}}}}

\newcommand{\U}{\ensuremath{\operatorname{\mathsf{U}}}}
\newcommand{\gU}{\ensuremath{\operatorname{\mathsf{U}}}}
\newcommand{\SU}{\ensuremath{\operatorname{\mathsf{SU}}}}
\newcommand{\Spin}{\ensuremath{\operatorname{\mathsf{Spin}}}}

\newcommand{\T}{\ensuremath{\operatorname{\mathsf{T}}}}
\renewcommand{\S}{\ensuremath{\operatorname{\mathsf{S}}}}

\newcommand{\hiota}{\hat{\iota}}\newcommand{\hPi}{\hat{\Pi}}
\newcommand{\hGamma}{\hat{\Gamma}}

\def\con#1=#2(#3){#1 \equiv #2 \bmod{#3}}

\newcommand{\ml}{\langle}                     
\newcommand{\mr}{\rangle}                    

\newcommand{\tv}{\tilde{v}}
\newcommand{\tiota}{\tilde{\iota}}

\newcommand{\no}{\noindent}

\begin{document}
\date{\today}

\title[Nonnegatively curved 4-manifolds with circle symmetry]{A knot characterization and 1-connected nonnegatively curved 4-manifolds with circle symmetry}


\author{Karsten Grove}
\address{University of Notre Dame\\
       Notre Dame, IN 46556-4618}
\email{kgrove2@nd.edu}

\author{Burkhard Wilking}
\address{University of M\"{u}nster\\
       Einsteinstrasse 62\\
       48149 M\"{u}nster, Germany}
\email{wilking@math.uni-muenster.de}


\thanks{
}

\maketitle


It is known that the only (closed simply connected) positively curved 4-manifolds
with infinite isometry group (equivalently having circle symmetry) are $\Sph^4$ and $\CP^2$.
In the case of nonnegative curvature additionally only $\Sph^2\times\Sph^2$ and $\CP^2\#\pm\CP^2$ will occur.
Topologically this is due via Freedmann's work \cite{Fr} to Hsiang  - Kleiner \cite{HK} in positive curvature
and Kleiner, Searle - Yang \cite{Kl,SY} in non-negative curvature, and differentiably it follows via the Poincar\'{e}
conjecture by their work and work of Fintushel and Pao \cite{Fi}, \cite{Fi2}, \cite{Pa}.

Our main purpose here is to provide a classification of (closed) simply connected nonnegatively
curved 4-manifolds with an isometric circle action up to equivariant diffeomorphism. In particular, we have

\begin{main}\label{linear}
A closed positively curved $4$-manifold $M$ with an isometric $\S^1$ action is equivariantly diffeomorphic
to a linear action on $\Sph^4$, $\RP^4$ or $\CP^2$.
\end{main} 

This actually is a consequence of the following general

\begin{main}\label{torus}
A closed nonnegatively curved simply connected $4$-manifold $M$ with an isometric $\S^1$ action
is diffeomorphic to $\Sph^4, \CP^2, \Sph^2\times\Sph^2$, or one of $\CP^2 \# \pm \CP^2$, and the action
extends to a smooth $\T^2$ action.
\end{main}

All such $\T^2$ actions have been classified by Orlik and Raymond in \cite{OR}, and
it turns out that for each such action there is an invariant metric of nonnegative curvature.
In fact, in the cases of Euler characteristic 2, or 3, such an action is linear, and if the
Euler characteristic is 4, any such $\T^2$ action can be obtained as the induced quotient action
on a $\T^2$ quotient of the standard $\T^4$ action of $\Sph^3 \times \Sph^3$ \cite{GK}. 

\medskip
Our proof uses Alexandrov geometry and the solution of the Poincar\'{e} conjecture in essential ways
(cf. Remark 5.6 (b) though). The point of departure for the latter is the simple fact that the orbit space
$M^*=M/\S^1$ is a simply connected topological 3-manifold, hence $\Sph^3$ (or $\Disc^3$). It is intriguing,
that in the presumably more complicated case of Euler characteristic 4, there is more geometric rigidity and
as a consequence, the Poincar\'{e} conjecture is not needed here (see section 4). In other words, we offer a
purely geometric proof of the equivariant classification when the Euler characteristic is four.

\medskip
By work of Pao \cite{Pa} it is known that there are $\gS^1$-actions on $\Sph^4$ such that the the singular 
set of the orbit space $\Sph^4/\gS^1$ is given by a knotted closed curve and in fact 
any locally flat knot can be realized in this way. 
Hence our equivariant classification relies on showing that there can be no knotted circles, $c$
in the singular set of $M^*$. For this we consider the canonical two fold branched cover of $M^*$ branched along
a circle in the singular set  $c$, denoted by $M^*_2(c)$. The following characterization is pivotal 

\begin{main}\label{knot}
Let $c$ be a locally flat embedded $\Sph^1$ in $\Sph^3$ and let $\Sph^3_2(c)$ denote the
corresponding canonical 2-fold branched cover. Then $\Sph^3_2(c)$ is the 3-sphere if $c$ is unknotted,
and otherwise $\pi_1(\Sph^3_2(c))$ is infinite or has order at least $3$. 
\end{main} 

The Alexandrov geometry of the orbit space $M^*= M/\S^1$ is already instrumental for the topological classification
alluded to above. Our use of the above knot characterization is based on the
key observation that also $M^*_2(c)$ is an Alexandrov space with the same lower curvature bound as that of $M^*$. When $
M^*$ has no boundary, i.e., $M^*$ is homeomorphic to the 3-sphere, we prove that any closed curve $c \subset M^*$ consisting
of non-principal orbits, is unknotted when $M^*$ has nonnegative curvature in distance comparison sense. A curve $c$
consisting of non-principal orbits provides an example of a so-called \emph{extremal set} in the Alexandrov space $X$. 

\bigskip

This naturally leads to the following question: Which knots $c$ in $\Sph^3$ can arise as extremal subsets
when $\Sph^3$ is equipped with the structure of an Alexandrov space with non-negative curvature.
We provide a characterization which
via the equivariant Poincar\'{e} conjecture leads to the following

\begin{main}
A knot $c$ in $\Sph^3$ can be an extremal set relative to a nonnegatively curved Alexandrov metri
c on the sphere if an only if it is a so-called spherical Montesinos knot of \emph{Cyclic}, \emph{Tetrahedral}
or \emph{Icosahedral} type. In particular all these knots arise as singular 
sets with respect to a constant curvature 1 orbifold metric on $\Sph^3$.
\end{main}

It is our pleasure to thank Darryl McCullough for pointing out that our characterization of extremal knots
indeed describes a subclass of the spherical Montesinos knots, and for providing the reference \cite{Sa} for us.
We also thank Fernando Galaz-Garcia and Alexander Lytchak for helpful comments.


\section{A knot characterization}

In this section we will see that Theorem C
in the introduction is a simple consequence of the solution of the (equivariant) Poincar\'{e} conjecture.  

\smallskip

If $c$ is a closed embedded smooth circle in $\Sph^3$, it is well known that the complement of $c$
admits a canonical 2-fold cover corresponding to the 2-fold cover of the normal circle to $c$.
By gluing back $c$ we get the 2-fold branched cover $\Sph^3_2(c)$ alluded to in the introduction. 

\smallskip

If $c$ is the unknot it is clear that $\Sph^3_2(c)$ is the 3-sphere.

\smallskip

Now suppose $c$ is knotted (the subsequent construction works without this assumption). Choose an orbifold metric on
$\Sph^3$ with normal cone angle $\pi$ along $c$. This can be done, e.g., by choosing a metric invariant
under the antipodal map of the normal bundle in a small tubular neighborhood of $c$, taking the quotient by this isometric
involution and gluing it back via a partition of unity. The induced metric $g$ on $\Sph^3_2(c)$ is Riemannian and $c$ is a 
geodesic in $\Sph^3_2(c)$ fixed by a global isometric involution $\iota$, the covering transformation on $\Sph^3_2(c) - c$.

The following is crucial 

\begin{lem}
If $\pi_1(\Sph^3_2(c))$ is finite, then $\Sph^3_2(c)$ admits a (positive) constant curvature metric in which $\iota$ remains an isometry.
\end{lem}

This is a consequence of the ``equivariant" Poincar\'{e} conjecture due to Dinkelbach and Leeb \cite{DL}. 
In particular, if $\pi_1(\Sph^3_2(c))$ has order two $\Sph^3_2(c)$ must be $\RP^3$ and the involution $\iota$ is linear.
This is already a contradiction, since such an $\iota$ will have two circles as fixed point set. 
This proves Theorem \ref{knot} modulo the following

\begin{lem}[Smoothing]\label{smoothing}
Any locally flat embedded circle $c$ in $\Sph^3$ can be smoothed.
That is, there is a homeomorphism of  $\Sph^3$ whose image of $c$ is a smooth submanifold.
\end{lem}

This, however, is an immediate consequence of the fact that any topological 3-manifold has a unique
smooth structure up to diffeomorphism, combined with the observation that  
$c$ has a neighborhood homeomorphic to $\Sph^1 \times \Disc^2$, which is an immediate consequence of \cite{Ki}.


\section{Orbit spaces and branched covers}

We begin this section with an analysis of the orbit space $M^* = M/\S^1$  with its induced orbital metric.
We denote by  $M^{\S^1}$  the fixed point set of the action. The projection map $M \to M^*$
 is a submetry and $M^*$ is a 3-dimensional Alexandrov space with nonnegative curvature 
 (positive if $M$ has positive curvature). Unless otherwise explicitly stated we assume throughout that $M$ is
 a closed simply connected $4$-manifold of nonnegative sectional curvature. In \cite{Kl} and \cite{SY}
 it was shown that the Euler characteristic of $M$, $\chi(M) = \chi(M^{\S^1})$ is at most $4$
 ($3$ if $M$ has positive curvature, see \cite{HK}). Note that $M^{\S^1}$ is also naturally a subset of $M^*$.

The orbit space has non-empty boundary $\partial M^*$ if and only if $M^{\S^1}$ is 2-dimensional.
In this case, a 2-dimensional component of $M^{\S^1}$ is also a component of $\partial M^*$ and $M$ is
said to be \emph{fixed point homogeneous}. By the soul theorem for orbit spaces, $\partial M^*$ has at most
two components, and in the case of two components, $M^*$ is isometric to the product of an interval and
a boundary component. A complete classification of the possible actions in this fixed point homogeneous case
was done for positive curvature in \cite{GS}, and for nonnegative curvature in \cite{Ga} and \cite{GK}. 

\medskip

Since the claims in our main theorem have been proved already in the fixed point homogeneous case, 
we assume from now on that $M^*$ has no boundary, i.e., $M^{\S^1}$ consists of 2, 3 or 4 isolated points. 

The isotropy representation at an isolated fixed point, $p \in M^{\S^1}$ has the form 
$e^{i\theta} (z_1, z_2) = (e^{i k \theta} z_1, e^{i \ell \theta} z_2)$, where $T_pM$ has been identified
with $\C^2$, and $k \ge \ell \ge 1$ are relatively prime. In particular, the action on the unit sphere $\Sph^3$ of
$T_pM$ is either free, has one isotropy group $\Z_k$, or two isotropy group $\Z_k$ and $\Z_{\ell}$. The 
corresponding orbit space $\Sph^3/\S^1 =:\Sph_{k,\ell}$ is isometric to the space of directions at $p \in M^*$,
is a (singular) surface of revolution and topologically the 2-sphere. In particular $M^*$ is a topological manifold,
and since it is clearly simply connected, the Poincar\'{e} conjecture tells us that:

\begin{lem}
$M^*$ is the 3-sphere.
\end{lem}

 Metrically it is important to note that

\begin{center}
$\Sph_{1,1} = \Sph^2(1/2)$ 
\end{center}

\no is the standard 2-sphere with radius 1/2, and 

\begin{equation}\label{space of direction}
\Sph_{1,1} \ge \Sph_{k,1} \ge \Sph_{k,\ell}
\end{equation}

\no in the sense that there are natural ($\S^1$ equivariant) distance decreasing maps to the smaller space.
The bound 4 (3 in positive curvature) on the number of fixed points is simply achieved by the observation that
there can be at most 4 (resp. 3) points in $M^*$ as singular as $\Sph^2(1/2)$ when the curvature in nonnegative
(resp. positive), cf. Lemma \ref{lem: four}. 

\medskip

A stratum in the orbit space $M^*$ corresponding to points with isotropy group $\Z_n$ forms a geodesic arc $\gamma$
whose closure joins two different isolated fixed points. The space of directions at a point
of $\gamma$ is isometric to the spherical suspension of a circle of length $2\pi/n$, i.e., to

\begin{center}
$\Sph^2(1)/\Z_n =: \Sph_n$
\end{center}

  Note that the spaces of directions with two singular points, i.e., $\Sph_{k,\ell}$ and $\Sph_n$, with $\ell, n \ge 2$,
  have natural two-fold  branched
   covers along the two singular points, denoted by $\Sph_{k/2,\ell/2}$ and $ \Sph_{n/2}$ respectively.
  

Also note that at most two geodesic strata corresponding to finite isotropy groups can end at a given fixed point,
then making a right angle.
The union of closures of two, three, or four such geodesic strata may form a closed curve $c$ in $M^*$.
Our main objective is to show that when this happens, such a closed curve in the 3-sphere $M^*$ is not knotted.
The following is a key observation

\begin{lem}[Branched Cover]\label{alex branching}
Suppose $c$ is a closed curve in $M^*$ formed by the closure of geodesic strata corresponding to finite isotropy groups. Then the two-fold branched cover $M^*_2(c)$ is an Alexandrov space with the same lower curvature bound as $M^*$. Moreover, the space of directions at any point $p \in c$ of $M^*_2(c)$ is the canonical two-fold branched cover of the corresponding space of directions of $M^*$. In particular, for any point $p\in c\cap M^{\gS^1}$ the space of directions satisfies $\Sigma_pM^*_2(c)\le \Sph^2(1/2)$. 
\end{lem}

\begin{proof}
We let $c_2\subset c$ denote the closed subset of those orbits whose isotropy group 
is not $\Z_2$. Notice that $M^*_2(c)$ is a smooth Riemannian manifold 
in a neighbourhood of $c-c_2$. 
It is well known that the isotropy group at interior points of a minimal geodesic between two orbits is the same as the isotropy group of the whole geodesic. In particular, the set of regular points of the orbit space corresponding to principal orbits is a convex set. For the same reason in our case, $M^* - c$ is convex. Thus it suffices to see that also $M^*_2(c) - c_2$ is convex in $M^*_2(c)$, since any geodesic triangle in $M^*_2(c)$ is the limit of geodesic triangles in $M^*_2(c) - c_2$. 

To prove the convexity claim above it suffices to see what the geodesics in $M^*_2(c)$ are. Here only those emanating at points in $c$ are an issue. By construction, however, it is clear that for each geodesic emanating at points of $c$ in $M^*$ there are exactly two emanating from the same point in $M^*_2(c)$. This shows that the spaces of directions at points along $c$ are as claimed in the lemma, and also shows that minimal geodesics between points of $M^*_2(c) -c$ can cross $c$ only at points with isotropy $\Z_2$ when viewed in $M^*$. 

Finally, 
if $k,\ell\ge 2$ are relatively prime, then it is easy to see $\Sph_{k/2,\ell/2}\le \Sph_{1,1} =\Sph^2(1/2)$.
\end{proof}

To use Theorem C we also need the following

\begin{lem}[Local Flatness]\label{flat}
A closed curve $c$ consisting of singular points of $M^*$ is a locally flat 1-dimensional submanifold of $M^*$.
\end{lem}

\begin{proof} We will denote the orbit of a point $p$ by $p^*\in M^*$.
It is immediate from the slice theorem that the exponential map at any $p \in M$ with $p^*\in c$ induces a
homeomorphism from a sufficiently small ball in the tangent cone at $p^*$ onto the corresponding ball 
centered at $p^*$. Since both $\Sph_{k,\ell}$, and $\Sph_n$ with $\ell, n\ge 2$ admit homeomorphisms 
to $\Sph^2(1)$ taking the pair singular points to a pair of antipodal points the claim follows. 
\end{proof}

We are now in position to prove our key result

\begin{thm}[Unknot]
Let $M$ be a simply connected nonnegatively curved 4-manifold with an isometric $\S^1$ action with isolated fixed points only. If $c$ is a circle in $M^*$ consisting of non-regular points, then $c$ is unknotted, there is at most one such curve, and all fixed points are on the curve forming a biangle, triangle, or quadrangle corresponding to 2, 3, or 4 fixed points.
\end{thm}

We also recall the following 

\begin{lem}\label{lem: four} A three dimensional nonnegatively curved Alexandrov space $A$ has at most four points 
for which the space of directions is not larger than $\Sph^2(1/2)$.
\end{lem}

\begin{proof} Suppose, on the contrary, we can find $5$ distinct points $p_1,\ldots, p_5$
with $\Sigma_{p_i}A\le \Sph^2(1/2)$.
 Choose for any pair of points $\{p_i,p_j\}$ in $\{p_1,\ldots,p_5\}$ a minimal geodesic
$c_{ij}$ and consider the corresponding $10$ geodesic triangles. 
The sum of the 30 angles adds up to at least $10\pi$.
 
We may assume after reordering that the sum of the $6$ angles
based at $p_1$ is at least $2\pi$. 
Next consider the initial 
directions of the chosen  minimal geodesics based at $p_1$. 
They are four distinct points $v_1,\ldots,v_4$ in the space of directions $\Sigma_{p_1}A$. 

Using that $\Sigma_{p_1}A\le \Sph^2(1/2)$
it is clear that the distances of any three points in $\{v_1,\ldots,v_4\}$ 
 add up to at most $\pi$. 
This of course shows that the sum of the $6$ angles is indeed equal to 
$2\pi$. From the equality discussion it is now easy to deduce 
that after possibly renumbering $v_1,\ldots,v_4$
we can assume $d(v_1,v_2)=d(v_3,v_4)=\pi/2$.  
Moreover, it is clear that  the point $p_1$ was arbitrary. 
Thus there are exactly $10$ right angles
and in each of the $10$ triangles in $A$ exactly one angle equals $\pi/2$. 

We may assume now that $d(p_1,p_2)=\min_{i\neq j}\{d(p_i,p_j)\}$. 
We choose $p_i\in \{p_3,p_4,p_5\}$ such that neither the angle between 
$c_{12}$ and $c_{1i}$ nor angle between $c_{12}$ and $c_{i2}$ 
is equal to $\pi/2$. 
This implies that the angle between $c_{1i}$ and $c_{2i}$ based at $p_i$
equals  $\pi/2$. 
Because of equality in Toponogov's theorem, 
$d(p_1,p_2)^2=d(p_1,p_i)^2+d(p_2,p_i)^2$ -- contradicting our above choice.
\end{proof}

\begin{proof} [Proof of the Unknot Theorem]

As we know $M$ has Euler characteristic 2, 3 or 4 corresponding to 2, 3, or 4 isolated fixed points of $\S^1$. Since the points of $c$ that correspond to fixed points of $\S^1$ remain more singular than $\Sph^2(1/2)$ according to our description above, and those possibly outside $c$ keep their size but double in numbers when we pass to the two-fold branched cover $M^*_2(c)$, we know that there are at most 4 (resp. 3 in positive curvature) such points also in $M^*_2(c)$ as well as in the universal cover of 
$M^*_2(c)$. If $c$ where knotted we could use
Theorem C to see that the number of singular points would at least triple if we pass to the universal cover. 

Thus $c$ is unknotted.
Moreover, 
applying Lemma~\ref{lem: four} to  $M^*_2(c)$ we deduce
that there is at most one such curve $c$, and all the fixed points are on this curve $c$, 
except possibly when there are three fixed points, the curvature is only nonnegative
and the singularities of $M^*$ are given by one isolated fixed point 
and two fixed points contained in a singular (unknotted) biangle $c$.
 It is a simple consequence of \cite{Fi}, Lemma 5.1 that no circle action on $\CP^2$ has an orbit space of this type.
\end{proof}

By Theorem 7.1 of Fintushel, \cite{Fi}  Theorem B  and hence also Theorem A when $M$ is simply connected follows.

In the next section we will give a direct geometric proof adapted to our assumptions, which also yields the extension to a $\T^2$ action when the Euler characteristic is four.


\section{Double Disc Bundle Decomposition}

In this section we will use the Poincar\'{e} conjecture to analyze $M$ with its circle action further.
Specifically we will use the fact that $M^*$ is the 3-sphere, and the Unknot Theorem of the previous section
to decompose $M^*$ in a specific way into two 3-discs respecting the strata, which in turn will yield a decomposition
of $M$ into two invariant disc bundles over points and or 2-spheres.
We assume that $M$ and hence $M^*=M/\gS^1$ is simply connected and nonnegatively curved, but see Remark~\ref{rem: rp4} below.

Consider the orbit space $M^*$. By the unknot theorem, the singular set in $M^*$ either forms a closed unknotted curve $c$, 
or we can extend the singular set to a closed embedded unknotted curve also denoted by $c$, such that any two arcs of $c$
make the maximal angle $\pi/2$ at each fixed point, and are geodesics near each fixed point. (We chose such a extension
only to make all arguments uniform). Note that the inverse image of each arc $\sigma$ of $c$ joining two fixed points
form a smooth invariant 2-sphere $\Sigma$ in $M$, actually the fixed point set of a finite isotropy group
(or a component thereof)
if the arc is a geodesic strata. To make it clear, $c$ is a right angled biangle, triangle or
quadrilangle/rectangle corresponding to the action having 2, 3 or 4 fixed points respectively.

The decomposition is now achieved as follows. Let $A$ and $B$ be invariant smooth ``dual" submanifolds
of $M$ corresponding either to the two fixed points when $c$ is a biangle,
one fixed point and the inverse image of the opposite edge when $c$ is a triangle, and the inverse image
of two opposite edges when $c$ is a rectangle. In a small $\epsilon$ neighborhood $U$ of the inverse image $C$
of $c$ we construct a smooth $\S^1$-invariant horizontal vector field $V$ which is ``normally radial" near $A$ and $B$
and tangential to the inverse image of the remaining two edges of $c$. This descends to a ``smooth" 
vector field $\bar V$ in the $\epsilon$ neighborhood $U^*$ of $c$ in $M^*$  which is ``normal" near the image
of the boundaries of the tubular neighborhoods of $A$ and $B$ and for which the remaining two edges of $c$
are integral curves. Using the fact that the $\epsilon$ neighborhoods of the images of $A$ and $B$
are 3-balls as are their complements, and the fact that $c$ is unknotted, it follows that $\bar V$ can be 
extended to a smooth nonvanishing vector field on the complement of $U^*$ 
respecting this ball decomposition of $M^*$. The extension of $\bar V$ obviously uniquely lifts to
an invariant extension of $V$ providing the desired decomposition of $M$ into tubular neighborhoods of $A$ and of $B$.

When $A$ is a point this immediately yields a proof of Theorem A.

 In the remaining cases $\chi(M) = 4$, $c$ is a ``rectangle", and $A$ and $B$ are 2-spheres.
 It is easy to see that the vector field $\bar V $ on $M^*$ can be chosen so that 
 the flow lines emanating from each point of one edge will meet at a point of the other edge
 to form a 2-sphere unless the points are vertices of the rectangle,
  in which case there is only one flow line. There is an $\S^1$ action on
  $M^*$ preserving these spheres with orbit space a 2-dimensional rectangle. This action clearly
  lifts to an action on $M$ whose orbits near $A$ and $B$ are the normal
   circles in a tubular neighborhood. It follows that this lift commutes with the given isometric $\S^1$ on $M$. 
 
 \begin{rem}\label{rem: rp4}
Note that if $M$ in Theorem A is not simply connected it has fundamental group $\Z_2$ by the Synge Theorem.
Consider the lifted $\S^1$ action on the universal cover $\tilde M$. It follows that $M$ has
Euler characteristic $1$, and we conclude that this action on $\tilde M$
is either the linear action on $\Sph^2 * \Sph^1$ with fixed point set $\Sph^2$
(fixed point homogeneous), or the suspension of a linear almost free action on $\Sph^3$ (as above).
In the first case our claim follows directly as in \cite{GS}. In the second case,
the covering group $\Z_2$ interchanges the two fixed points and preserves a 3-sphere in $\tilde M = \Sph^4$ 
invariant under the $\S^1$ action as well. From the equivariant Poincar\'{e} conjecture it follows that
indeed $\Z_2$ acts as the antipodal map and we are done.
\end{rem}


\section{Rigidity for Euler Characteristic Four}
The main aim of this section is to prove the following equality discussion in Kleiner's estimate of the Euler characteristic of a nonnegatively curved $4$-manifold 
with circle symmetry.
\begin{thm}\label{thm: rigid}\label{thm: euler 4} Let $(M,g)$ be a nonnegatively curved $4$-manifold of Euler characteristic four 
endowed with an isometric $\S^1$ action which has only isolated fixed points.
Then one of the following holds \begin{enumerate}
\item[a)] There is a totally geodesic flat torus $T^2$ which is horizontal with respect 
to the $\S^1$-action and projects to an embedded $2$-sphere $\Sph^2\subset M^*$
endowed with a flat orbifold metric. 
\item[b)] There are two closed intervals 
$I,J\subset \R$ and a submetry $\sigma\colon M^4\rightarrow I\times J$. 
The fibers of $\sigma$ are given by the orbits of a (not necessarily isometric)  $T^2$ action
that extends the given $\S^1$-action.
\end{enumerate}
Moreover b) holds if one of the following is true 
\begin{enumerate}
 \item[(i)] In  $M^*$  there are two fixed points which can be connected by more than one minimal geodesic.  
\item[(ii)] There is one fixed point whose space of direction in $M^*$ has two orbifold 
singularities. 
\end{enumerate}

\end{thm}

For the proof we need
\begin{lem}\label{lem: euler 4} Let $p_1,p_2,p_3,p_4$ be the four isolated fixed points 
and let $c_{ij}$ be a minimal geodesic between $p_i$ and $p_j$.   \begin{enumerate}
\item[a)] If the minimal geodesic in $M^*$ between $p_1,p_3$ 
is not unique then the minimal geodesics 
$c_{12}$, $c_{14}$ $c_{23}$, $c_{34}$ are unique and 
$\angle (c_{12},c_{14})=\angle (c_{34},c_{23})=\pi/2$. 
\item[b)] If $\angle (c_{12},c_{14})=\pi/2$ then $c_{24}$ is not unique. 
In fact
the initial directions of all minimal geodesics from $p_2$ to $p_4$ 
form a circle.\end{enumerate}
\end{lem}

It will be clear from the proof of b) that the
number $\angle (c_{23},c_{24})$ is independent of 
the choice of $c_{24}$.  The fact that the initial directions form a circle implies 
that this is indeed the only constraint for initial directions 
of minimal geodesics from $p_2$ to $p_3$.

\begin{proof}[Proof of Lemma~\ref{lem: euler 4}]
{\em a)} For any three points 
the sum of the angles is at least $\pi$. 
Thus the $12$ angles add up to at least $4\pi$. 
For each fixed point $p_i$ the three angles based at $p_i$ form 
a triangle in $\Sigma_{p_i}M^*\le \Sph^2(1/2)$. 
Thus the three angles can add to up to at most $\pi$. 
Clearly equality must hold everywhere.
Let $v_2$, $v_3$, $v_4$  be the initial vectors of $c_{12}$, $c_{13}$, $c_{14}$ 
and let $\tilde{v}_3$ denote the initial direction of a different minimal
from $p_1$ to $p_3$. 
The sum of the angles is in each case $\pi$. 

Recall that distance of three points  in $\Sph^2(1/2)$ only add up to $\pi$ 
if the three points are on a great circle but not on an open semicircle.
Moreover $d(v_2,v_3)=d(v_2,\tv_3)$ because both numbers are 
equal to the angle in the comparison triangle of $p_1,p_2,p_3$.
Using that $\Sigma_{p_i}M^*\le \Sph^2(1/2)$ 
this readily implies
$\angle(v_2,v_4)=\pi/2$.
If $v_2$ would not be unique one could prove similarly $\angle(v_3,v_4)=\pi/2$ 
which is clearly impossible.

{\em b)} Let $c_{12}^h$ and $c_{14}^h$ be arbitrary lifts of $c_{12}$ and $c_{14}$. 
Using $\angle (c_{12},c_{14})=\pi/2$ we deduce that 
the initial directions of $c_{12}^h$ and $c_{14}^h$ can be chosen arbitrarily in 
two  two-dimensional orthogonal subspaces. 
In particular for any choice $\angle (c_{12}^h,c_{14}^h)=\pi/2$. 
By the equality discussion in 
Toponogov theorem there is a flat triangle  in
$M$ which has $c_{12}^h$ and $c_{14}^h$ in its boundary. 
The remaining side is a minimal geodesic $c_{13}$ from $p_2$ to $p_4$ in $M$. 

Moreover the triangle must be horizontal. 
It is clear that this gives a two dimensional ($\T^2$ ) family  of minimal geodesics from 
$p_2$ to $p_4$ and thus $c_{24}$ is not unique in $M^*$.
\end{proof}

\begin{proof}[Proof of Theorem~\ref{thm: euler 4}]
The lemma implies in particular that if there is one right angle 
then there are exactly four and the
configurations forms a unique  rectangle and with many  choices for the diagonals. 
From the proof it is then easy to see that any geodesic leaving 
perpendicular to one side of the rectangle 
meets the opposite side of the rectangle after the exact same time. 
Therefore, 
if $E_1$ and $E_2$ are two  sides of the triangle meeting in 
$p_1$, then $M^*\rightarrow \R^2$, $p\mapsto (d(E_1,p),d(E_2,p))$ 
is a submetry onto its image.

This proves the theorem as long as one of the geodesics is not unique.
In the remaining case all geodesics are unique and there are no right 
angles. Then each of the four triangles can be filled uniquely with 
a flat convex set which forms a totally geodesic $\Sph^2$ in $M^4/\S^1$. 
Using the uniqueness of the minimal geodesics in $M^*$ between fixed points it is then easy to 
see that a horizontal lift of $\Sph^2$ 'closes up' to a compact torus. 

Finally if there is one fixed point $p$ for which the space of direction 
$\Sigma_p(M^*)$ is given by $\Sph_{k,\ell}$ with $k,\ell \ge 2$, then 
we consider the three initial directions $v_1$, $v_2$, $v_3$ of 
minimal geodesics to the other fixed points. 
Using that $\sum_{i<j}\angle(v_i,v_j)=\pi$ we then deduce that
the two orbifold singularities in $\Sigma_pX=\Sph_{k,\ell}$ must be contained in 
$\{v_1,v_2,v_3\}$ and thus a right angle occurs.
 \end{proof}

\begin{rem}  


(a) In the situation of Theorem~\ref{thm: euler 4} a) one can use the soul theorem to see (without using the Poincar\'e conjecture) 
that either side of the $\Sph^2$ is given by a $3$-disc, see also Remark~\ref{rem: final} b).

(b) There are nonnegatively curved metrics  (M\"uter metrics) on $\Sph^2\times \Sph^2$ with an $\S^1$-symmetry 
for which $ \Sph^2\times \Sph^2/\S^1$ has positive curvature outside a set of 
codimension 1. 

(c)  Theorem~\ref{thm: rigid} remains valid if $M$ is replaced by a nonnegatively curved $4$-dimensional orbifold 
with Euler characteristic four. Recall that by a result of Martin Kerin [Ke, Theorem 2.4]
there are a lot of $4$-dimensional orbifolds of Euler characteristic four which have positive sectional curvature 
on an open dense set. These orbifolds  have an $\S^1$ symmetry with four isolated fixed points.
\end{rem}

\section{Knots arising as extremal subsets in $3$-dimensional 
Alexandrov spaces.} 

Throughout this section we let $X$ denote the 3-sphere endowed with an arbitrary Alexandrov metric of nonnegative curvature, and $c \subset X$ an extremal subset homeomorphic to $\Sph^1$.  Recall that $c$ being extremal means that for any point $p\in X\setminus c$, any point $q\in c$ with $d(p,q)=d(p,c)$ is a critical point of the distance function $d(p,\cdot)$. 
 Such sets play an important role in Alexandrov geometry cf. \cite{Petrunin} for further information.

The question we answer in this section is: Which knots 
can arise in this way? We will first show that the question is equivalent 
to asking for which knots $c$ in $X$ is the two fold branched cover $X_2(c)$ a spherical space form. 

To see this we need generalizations of both Lemma \ref{flat} and Lemma \ref{alex branching}.

\begin{lem}
An extremal closed curve $c \subset X$ is a locally flat 1-dimensional submanifold.
\end{lem}

\begin{proof}
As in the proof of \ref{flat} it suffices to see that for each point of $c$, a sufficiently small ball is homeomorphic to the corresponding ball in its tangent cone by a homeomorphism taking the intersection of $c$ with the ball to the corresponding extremal curve in the tangent cone. The latter is of course the intersection of the ball with the cone of the space of directions of $c$, which is a pair of extremal points in the space of directions of $X$ at the point. The claim is a direct consequence of the Relative Stability Theorem due to Kapovitch \cite{Ka} (Theorem 9.2) via scaling. 
 \end{proof}

\begin{lem} $X_2(c)$ is an Alexandrov space with nonnegative curvature.
\end{lem}

\begin{proof} 
 For $a\in c$, the tangent cone $T_a X_2(c)$ 
 exists and is a two fold branched cover of $T_a X$.
 It is easy to see that the space of directions $\Sigma_aX_2(c)$ of $X_2(c)$ at $a$ 
 is an Alexandrov space of curvature $\ge 1$ whose projection to
 the space of directions of $X$ at $a$ is a  two fold branched cover. 
The branching locus of the cover in $\Sigma_aX_2(c)$ is given by two points corresponding to 
 the two directions in $T_ac$.
 
 Clearly geodesics in the two fold branched cover $\sigma\colon X_2(c)\rightarrow X$ are non-branching 
and even more two minimal  geodesics $c_1,c_2\colon [0,1]\rightarrow X_2(c)$
 with the same initial direction coincide.\\[1ex] 
{\bf Step 1.} For any hinge in $X_2(c)$ based at a point $a\in c$ Toponogov's theorem holds.\\[1ex]
 We consider the gradient exponential map gexp$_a\colon T_aX\rightarrow X$ 
 due to Petrunin \cite{Petrunin}. It is $1$-Lipschitz and equal to the usual 
 exponential map at all points where the latter is defined. Since the gradient of the distance 
 function $d(q,\cdot)$ is tangential to $c$ along $c\setminus q$, it follows that
 if $v\in \Sigma_a X$ satisfies  gexp$(t_0v) \in c$ for some 
 $t_0>0$ then gexp$(tv)\in c $ for all $t>t_0$. 
  This in turn shows 
 that for any  lift $\hat v\in \Sigma_aX_2(c)$ of $v$ we can find a unique continuous lift 
 of gexp$(tv)$ ($t\ge 0$) to a continuous curve $\widehat{\mbox{gexp}}(tv)$. 
 It is then easy to see that $\widehat{\mbox{gexp}}_a\colon T_aX_2(c)\rightarrow X_2(c)$ is 1-Lipschitz. In fact one can distinguish between two cases. If $w\in T_aX_2(c)$ with $\widehat{\mbox{gexp}}(w)\notin c$ then $\widehat{\mbox{gexp}}$ is 1-Lipschitz in a neighbourhood of $w$ 
 since it is a lift of gexp.
 If $w\in T_aX_2(c)$ with $\widehat{\mbox{gexp}}(w)\in c$ then 

 \[
 d(\widehat{\mbox{gexp}}(w),\widehat{\mbox{gexp}}(v))
 =d(\mbox{gexp}(\sigma_*w),\mbox{gexp}(\sigma_*v)) 
 \le d(\sigma_*v,\sigma_*w)\le d(v,w).
 \] 
Clearly $\widehat{\mbox{gexp}}_a$ coincides with the usual exponential map for all points where the latter is defined. 
If $v,w\in T_aX_2(c)$ are the initial vectors of a hinge in $X_2(c)$, then the Euclidean comparison 
hinge is isometric to the hinge spanned by $v$ and $w$ in $T_aX_2(c)$ (based at the cone point). 
Since $\widehat{\mbox{gexp}}_a$ is 1-Lipschitz it maps the opposite side of the hinge to a shorter curve and the comparison follows.\\[1ex]
{\bf Step 2.} Let $q\in X_2(c)\setminus c$ and $S$ the set of all 
points $s\in X_2(c)$  for which a minimal geodesic from $q$ to $s$ 
passes through $c$. Then $S$ has measure zero.\\[1ex] 
 We let $c^h$ denote the subset of all points $s\in c$ 
 such that the minimal geodesic $\gamma_{qs}$
 from $q$ to $s$ continues to be minimal if we extend it by $h>0$.
 
 There is a map $f_h\colon c^h\rightarrow S$ which maps 
 $s=\gamma_{qs}(d(q,s))$ to the point $\gamma_{qs}(d(q,s)+h)$. 
 
 We finish the proof by showing that this map is locally Lipschitz. 
 Put $\eps_0=\tfrac{1}{3}\min \{1,h,d(q,c)\}$.
 More precisely we plan to show that for two points $s,s'\in c^h$ with $d(s,s')\le \eps_0$ 
 we have $d(f_h(s),f_h(s'))\le  \tfrac{4h}{\eps_0}d(s,s')$.

 Let $v\in T_sX_2(c)$ be the initial vector of the minimal geodesic to $f_h(s)$, 
 $u\in T_sX_2(c)$ the initial vector of the minimal geodesic to $q$ and $x\in T_sX_2(c)$ the initial 
 vector of a minimal geodesic  from $s$ to $s'$.  
 Of course the minimal geodesic in $T_sX_2(c)$ passes through the cone point and therefore the 
 triangle spanned by $u,v$ and $x$ is Euclidean. 
 By assumption we know that $|x|=d(s,s') \le \eps_0\le \tfrac{1}{3}\min\{|u|,|v|\}$.
 This in turn implies that we can estimate the defect in the triangle in equality 
 $d(u,x)+d(v,x)-(|u|+|v|)\le \tfrac{|x|^2}{\eps_0}$. Since the $g$ exponential map $\widehat{\mbox{gexp}}_s$
 is $1$-Lipschitz,
 we obtain $d(s',f_h(s))+d(s',q)-d(q,f_h(s))\le \tfrac{d(s,s')^2}{\eps_0}$.
 By Step 1, we can apply the hinge version of Toponogov's theorem to  the triangle $(f_s(h),q,s')$ to see that 
 that the angle between the minimal geodesic from $s'$ to $f_s(h)$ and the minimal 
 geodesic from $s'$ to $q$ is given by $\pi-\varphi$ with 
 \[
 \varphi \le \tfrac{2d(s,s')}{\eps_0}. 
 \]
Next notice that $\varphi$ is the angle between the minimal geodesic from $s'$ to $f_h(s)$ 
and the minimal geodesic from $s'$ to $f_h(s')$. 
Using Step 1 once more and $d(s',f_h(s))\le h+d(s,s')$ 
 we see that $d(f_h(s),f_h(s'))\le  \tfrac{2d(s,s')h}{\eps_0}+d(s,s')\le 4\tfrac{d(s,s')h}{\eps_0}$.\\[1ex]
{\bf Step 3.} Toponogov's theorem holds in $X_2(c)$.\\[1ex]
Notice that 
any geodesic in $X_2(c)$ is the limit of geodesics
which do not meet $c$.
It is well known that Toponogov's theorem is equivalent to saying 
that for all $q$ in $M$ the function 
$f(p)=d(p,q)^2$ satisfies $\mathrm{Hess}(f)\le 2$ in the sense of support functions, i.e., 
$f\circ \gamma(t)-t^2$ is concave for any unit speed geodesic $t\in [a,b]$. 

Clearly we may assume that $\gamma$ and $q$ are generic. 
Thus we may assume that $\gamma(t)$ does not meet the branching locus 
and (by Step 2) that the set of parameters $t$ for which a minimal geodesic from 
$q$ to $\gamma(t)$ intersects $c$ 
forms a set of measure zero. 

Since $X_2(c)$ is in a neighbourhood of the image of $c$ an Alexandrov space it follows that 
$\mathrm{Hess}(f)(\gamma(t))\le C$ for all $t\in [a,b]$ for some large $C$. 
Thus it suffices to prove that $\mathrm{Hess}(f)(\gamma(t))\le 2$ for all $t$ for which 
a minimal geodesic $\alpha$ from $q$ to $\gamma(t)$ does not meet the branching locus $c\subset X_2(c)$.
This in turn one can establish analogous to the 
usual  globalization theorem of 
Toponogov's comparison statement. 
\end{proof}

\medskip

We are now ready to analyze our question. As mentioned above our point of departure is the following

\begin{lem}  The two fold branched cover $X_2(c)$ of $X$ has finite fundamental group. 
And $H_1(X_2(c),\Z_2)=0$. 
\end{lem}

\begin{proof}

We view $H_1(X_2(c),\Z_2)$ as vectorspace over $\Z_2$.
Suppose for the moment it is not $0$-dimensional. 
Let $\iota$ denote the involutive isometry of $X_2(c)$
 given by the nontrivial deck transformation of the branched cover.
The automorphism $H_1(\iota)$ of $H_1(X_2(c),\Z_2)$ induced by $\iota$ 
then leaves some subspace $V\subset H_1(X_2(c),\Z_2)$ of codimension $1 $ 
invariant. Since $H_1(X_2(c),\Z_2)\cong \Pi/\ml \{g^2 \mid g\in \Pi \}\mr$ 
is a quotient of $\Pi$, 
this in turn shows that $\Pi$ contains a subgroup $\Pi'$ of index $2$ 
which is invariant under the natural action of $\pi_1(\iota)$.

Consider next the action of $\Pi$ on the universal cover 
$\widetilde{X_2(c)}$. 
We can also lift $\iota$ to an isometry $\tilde{\iota}$ of the universal cover.

Let $\Gamma$ be the group generated by $\tilde{\iota}$ and $\Pi$ 
and 
let $\hGamma$ be the normal subgroup generated by the conjugacy class 
of $\tilde{\iota}$ in $\Gamma$. 
Finally let $\Gamma'$ be the group generated by $\tilde{\iota}$ 
and the subgroup $\Pi'$ which is normalized by $\tilde{\iota}$. 
Since $\Gamma'$ has index two in $\Gamma$ it is normal and thus 
$\hGamma\subset \Gamma'\neq \Gamma$.

Using that the fixed point set of $\iota$ in $X_2(c)$ is connected we deduce 
that any element in the group $\Gamma\setminus \{e\}$ with a nontrivial fixed point 
set is in the conjugacy class of $\tilde{\iota}$. 
Thus
the group $\Gamma/\hGamma$ acts freely on $\widetilde{X_2(c)}/\hGamma$ 
and $X=\widetilde{X_2(c)}/\Gamma$ is not simply connected -- a contradiction.

Hence $H_1(X_2(c),\Z_2)=0$.  Using that $X_2(c)$ is a topological manifold with 
a nonnegatively curved Alexandrov metric we see that either 
$\pi_1(X_2(c))$ is finite, $X_2(c)$ is two fold covered by 
$\Sph^1\times \Sph^2$ or $X_2(c)$ is a flat Riemannian manifold.
Since  $H_1(X_2(c),\Z_2)$ is non-trivial for the latter two cases (see \cite{Wo}, Theorem 3.5.5)
we conclude that  $\pi_1(X_2(c))$ is finite.

\end{proof}

Using the equivariant elliptization conjecture we know that $X_2(c)$ is a spherical 
space form endowed with a linear involution whose fixed point set is a circle. 
In addition $X_2(c)$ is a $\Z_2$-homology $3$-sphere.

\begin{lem}\label{lem: space forms} \label{space forms}Let $\Sph^3/\Pi$ be a spherical space form which is also a $\Z_2$ homology sphere. 
Then there is a linear involution $\iota$ (unique up to conjugation) whose fixed point set is 
a circle. Moreover the underlying topological space of the quotient $\Sph^3/\ml \Pi,\iota\mr$ is 
  the $3$-sphere.
\end{lem}

\begin{proof} By assumption the
the abelianization $\Pi/[\Pi,\Pi]$ has no two torsion.
Recall that the finite groups $\Pi\subset \SO(4)$ that act freely on $\Sph^3$ 
 are conjugate to subgroups of $\gU(2)$. 
If we divide out the center of $\gU(2)$ we get a homomorphism 
$\Pi\rightarrow \SO(3)$ whose kernel is central in $\Pi$. 
By the classification of finite subgroups of $\SO(3)$ the image is either cyclic, dihedral 
or given as the orientation preserving symmetries of one of the platonic solids.

If the image is cyclic then $\Pi$ is abelian (as a central extension of a cyclic group)
and  since $\Pi$ acts freely this in turn implies that $\Pi$ itself is cyclic. 
Since the 
$\Pi/[\Pi,\Pi]$ has no two torsion, the image of $\Pi$ in $\SO(3)$ can not be dihedral 
and it can neither be given as the orientation preserving symmetries of a octagon (or the dual cube). 
 
This leaves us with three cases, $\Pi$ is either cyclic (Case 1), of tetrahedral type (Case 2) or of 
icosahedral type (Case 3). In the last two cases we will also use that the greatest common divisor 
of  the orders of the subgroups  $\Pi\cap \SU(2)$ and $\Pi\cap \mathrm{Center}(\gU(2))$ is at most $2$, 
as otherwise the action would not be free.

In all cases we establish first the existence 
of $\iota$ and it will be clear from the construction that 
$\iota$ fixes at least one circle.
In all cases this implies that the fixed point set in $\Sph^3/\Pi$ is equal to a circle 
because by a theorem of Floyd the total $\Z_2$-Betti number of the fixed point of $\iota$ 
is bounded by the the  $\Z_2$  Betti number of $  \Sph^3/\Pi$ which is $2$ by assumption.

{\em Case 1.} (Cyclic case) Here $\Pi$ is a cyclic group of odd 
order $m$. Let $\zeta\in \S^1$ be a primitive $m$-root of unity 
We may assume that for some  integer $p$ (prime to $m$) the action of  
$\Pi$ is generated by 
$\zeta\star (z_1,z_2)= (\zeta z_1,\zeta ^p z_2)$ 
where $z_1,z_2\in \C$ with $|z_1|^2+|z_2|^2=1$. 

We let $\tiota$ denote complex conjugation. 
Clearly $\tiota$ normalizes the action of $\Pi$. 
The group generated by $\Pi$ and $\tiota$ is a dihedral group 
which is generated by elements which are conjugate to $\tiota$. 
Since $\Pi$ is generated by elements with a nontrivial fixed point set 
the underlying topological space of 
$\Sph^3/\ml \tiota, \Pi\mr $ is simply connected
 and hence a $3$-sphere.
The involution $\iota\colon \Sph^3/\Pi\rightarrow \Sph^3/\Pi$ induced by $\tiota$
fixes at least one circle. 
As explained above this implies that the fixed point set is given by a circle. 
The uniqueness of $\iota$ follows from the fact that $\iota$ has to anticommute with the action of $\Pi$.

{\em Case 2.}(Tetrahedral case) We can assume $\Pi\subset \U(2)\subset \SO(4)$ 
We let $\hat{\Pi}\subset \S^3\times \S^1\subset \S^3\times \S^3= \Spin(4)$ denote the inverse 
image. We assume that $\S^1\subset \S^3$ is the image of the one parameter group $e^{i\varphi}$.

Then we can find a number $m$ coprime to $6$ 
and an integer $k\ge 0$ with $k\neq 1$ such that  
$\hat{\Pi}$ is the smallest group satisfying the following

\begin{enumerate}
\item[(i)] $\hat{\Pi}$ contains the cyclic subgroup of order $2m$ in $1\times \S^1$,
\item[(ii)] for one primitive $3^k$-th root of unity $\zeta\in \S^1$ (if $k=0$ this means $\zeta=1$),  
$\Pi$ contains the element $(\tfrac{1}{2}(1+i+j+k),\zeta)\in \S^3\times \S^1$
\item[(iii)] $\Pi$ contains the group $\{\pm 1,\pm i,\pm j,\pm k\}\times \{1\}\subset \S^3\times \S^1$.
\end{enumerate}

We put $\hiota\colon (\tfrac{1}{\sqrt{2}}(i+j),j)\in \S^3\times \S^3$. It is straightforward 
to check that $\hiota$ normalizes $\hPi$.
Clearly the image $\tiota$ of $\hiota$ in $\SO(4)$ and corresponds again to a complex conjugation 
and thus fixes a circle in $\Sph^3$. 
It is easy to see that the group $\Gamma:=\ml \Pi,\tiota\mr$ is generated by the conjugacy class 
 of 
$\tiota$ . 

Thus $\Sph^3/\ml \Pi,\tiota\mr$ is again a sphere.
As before the map $\iota\colon \Sph^3/\Pi\rightarrow \Sph^3/\Pi$ has the desired properties. 

The uniqueness follows from the fact that $\tiota$ normalizes $\Pi$ and that 
the conjugacy class of $\tiota$ generates $\ml \Pi,\tiota\mr$.

{\em Case 3.} (Icosahedral case) We can assume $\Pi\subset \U(2)\subset \SO(4)$ 
We let $\hat{\Pi}\subset \S^3\times \S^1\subset \S^3\times \S^3= \Spin(4)$ denote the inverse 
image. We assume that $\S^1\subset \S^3$ is the image of the one parameter group $e^{i\varphi}$. 
In this case $\hat{\Pi}$ is a product subgroup 
$\Pi=\Pi_1\times \Pi_2\subset  \S^3\times \S^1$, 
where $\Pi_1$ is the binary icosahedral group and $\Pi_2$ is a cyclic subgroup of order $2m$ 
with $m$ being an integer coprime to $30$.
The binary icosahedral group $\Pi_1$ is the group generated by 
$\pm i,\pm j,\pm k,\pm 1$, ${i+j}{\sqrt{2}}$, 
$\tfrac{1}{2}(1+i+j+k)$ and $i + \tfrac{1}{2}(1+\sqrt{5})j + \tfrac{2}{1+\sqrt{5}} k$. 
It has order 120 and the only normal subgroup is given by $\pm 1$. 

We put $\hiota=(j,j)$.  Clearly $\hiota$ normalizes
$\hPi$. We consider in $\ml \hPi,\hiota \mr$ the subgroup $\hat{\Gamma}$ 
generated by the conjugacy class of $\hiota$. 
The projection of $\hat{\Gamma}$ to $\S^3 \times\{1\}$ is given by 
a normal subgroup of $\Pi_1$ generated by the conjugacy class of $j$ in $\Pi_1$. 
Since $\Pi_1$ has no nontrivial  normal subgroups of order larger than $2$ we deduce
that $\hat{\Gamma}$ projects surjectively to $\Pi_1$. 
Using that the commutator group of $\Pi_1$ is equal to $\Pi_1$ 
and that the second commutator group of $\hat{\Gamma}$ is contained in $\S^3\times 1$ 
we see that $\hat{\Gamma}$ contains $\Pi_1\times \{1\}$. 
Moreover it is clear the projection  of $\hat{\Gamma}$ to $ \{1\}\times \S^3$ 
contains $\Pi_2$. In summary this proves 
$\hat{\Gamma}= \ml \hPi,\hiota \mr$. 
The image $\tiota$ of $\hiota$ in $\SO(4)$ is again a complex conjugation 
and one can finish the argument as before.

\end{proof}

\begin{cor} If $c$ is an closed embedded curved in $A=\Sph^3$ 
which is an extremal set with respect to a nonnegatively curved Alexandrov metric 
then we can find an orbifold metric $\Sph^3$ of constant curvature 1 such that 
the only orbifold singularity is of $\Z_2$ type and the singular locus is given by $c$. 
Moreover these orbifolds are classified in the proof of Lemma~\ref{space forms}.
\end{cor}

\begin{rem}\label{rem: final}

(a) In the special case that $\Gamma$ is a cyclic group of odd order, 
the corresponding knots arising are, by work of Seifert, the so called two bridge knots.
By Milnor two bridge knots in $\R^3$ are exactly those 
knots which for any $\eps>0$ are isotopic to knots with total curvature 
$\le 4\pi+\eps$. In all cases, the Hopf $\S^1$ action gives rise to a Seifert fibered structure on the space form $\Sph^3/\Pi = X_2(c)$ preserved by the involution $\iota$ induced from complex conjugation whose fixed point set is $c$. The knot $c$ in the 3-sphere $X_2(c)/\iota$ is explicitly described by Montesinos via the Seifert  invariants of the fibration (see, e.g., \cite{Sa}). In particular we mention that $c$ is the $(3,5)$ - \emph{torus knot} in the special case where $X_2(c)= \Sph^3/\Pi$ is the Poincar\'e homology sphere.

(b) In his PhD thesis which is still in preparation Wolfgang Spindeler shows that 
a quotient $X$ of a simply connected nonnegatively curved manifold by an isometric $\gS^1$ action is the Gromov Haussdorff limit 
of positively curved three manifolds. If $X$ contains a closed curve $c$ in its singular set then one can also find a sequence of smooth positively 
curved orbifold structures on $X$ (Gromov Hausdorff converging to the quotient metric) such that the only orbifold singularity 
is a $\Z_2$ singularity along $c$. In particular each of these metrics induces a smooth Riemannian metric on $X_2(c)$. 
This in turn allows   that the equivariant classification of the corresponding $4$-manifolds can also be proved without using the 
Perelman's solution of the Poincar\'{e} conjecture and the $\Z_2$-equivariant version. Instead one can then just refer 
to Hamilton's classification of positively curved $3$-manifolds.
\end{rem}

\providecommand{\bysame}{\leavevmode\hbox
to3em{\hrulefill}\thinspace}

\end{document}